\documentclass[reqno]{amsart}
\usepackage{amsmath}

\usepackage{amsfonts}
\usepackage{amssymb}
\usepackage{hyperref}
\usepackage{centernot}
\usepackage[a4paper,top=3.5cm,bottom=3cm,left=3.5cm,right=3.5cm]{geometry}

\begin{document}
\title[Positive solutions of $p$-Laplacian fractional differential equations]{Positive solutions of $p$-Laplacian fractional differential equations with fractional derivative boundary condition}
\author[F. Haddouchi]{Faouzi Haddouchi}
\address{
Faculty of Physics, University of Sciences and Technology of
Oran-MB, El Mnaouar, BP 1505, 31000 Oran, Algeria
\newline
And
\newline
Laboratoire de Math\'ematiques Fondamentales et Appliqu\'ees d'Oran (LMFAO). Universit\'e Oran1. B.P. 1524 El Mnaouer, Oran, Alg\'erie.}

\email{fhaddouchi@gmail.com; fouzi.haddouchi@univ-usto.dz}
\email{}
\email{}
\subjclass[2010]{34A08, 26A33, 34B18}
\keywords{Caputo fractional differential equations, $p$-Laplacian operator, positive solutions, fixed-point theorem, existence, cone}

\begin{abstract}
In this paper, we show some results about the existence and the uniqueness of the positive solution for a $p$-Laplacian fractional differential equations with fractional derivative boundary condition. Our results are based on Krasnosel'skii's fixed point theorem, the nonlinear alternative of Leray-Schauder type and contraction mapping principle. Three examples are given to illustrate the applicability of our main results.

\end{abstract}

\maketitle \numberwithin{equation}{section}
\newtheorem{theorem}{Theorem}[section]
\newtheorem{lemma}[theorem]{Lemma}
\newtheorem{definition}[theorem]{Definition}
\newtheorem{proposition}[theorem]{Proposition}
\newtheorem{corollary}[theorem]{Corollary}
\newtheorem{remark}[theorem]{Remark}
\newtheorem{exmp}{Example}[section]

\section{Introduction\label{sec:1}}

Fractional derivatives and integrals are proved to be more useful in the modeling of different physical and natural
phenomena. The $p$-Laplacian fractional boundary value problems related to nonlocal conditions have many applications in various fields, such as in the theory of heat conduction in materials with memory, non-Newtonian mechanics,
nonlinear elasticity and glaciology, combustion theory, population biology, nonlinear flow laws and so on.
For instance, when studying the steady-state turbulent flow with reaction, Bobisud (cf.\cite{Bobi}) introduced the differential equation
\[\big(\varphi_{p}\big(u^{\prime}(t)\big)\big)^{\prime}=f(t,u(t),u^{\prime}(t))\]
with an operator $\varphi_{p}(x)=|x|^{p-1}x$. This problem appears in the study of non-Newtonian fluids. It yields
the usual problem for diffusion in a porous medium when $p=1$, i.e., $\varphi_{p}(x)=x$.
The current analysis of these problems has a great interest and many methods are used to solve such problems.
Recently, the study of existence of positive solution to $p$-Laplacian fractional boundary value problems has gained much attention and is rapidly growing field. There are many papers concerning fractional differential equations with the $p$-Laplacian operator, see \cite{Hu, Khan3,Luca,Yupin, Nazim,Yunhong,Zhenlai,Hongling,Yang,Shen}.

In 2013, S. Ying et al. \cite{Ying} considered the existence criteria for positive solutions of the nonlinear
$p$-Laplacian fractional differential equation
  \begin{equation*}\label{}
 \begin{cases}
     \Big(\varphi_{p}\big(D^{\alpha}u(t)\big)\Big)^{\prime}=\varphi_{p}(\lambda)f(t,u(t),u^{\prime}(t)),\  t \in (0,1),\\
    k_{0}u(0) = k_{1}u(1), \   m_{0}u(0) = m_{1}u(1),\  u^{(r)}(0)=0,\ r=2,3,...,[\alpha],
  \end{cases}
    \end{equation*}
where $\varphi_{p}$ is the $p$-Laplacian operator, i.e., $\varphi_{p}(s) =|s|^{p-2}s$, $p>1$, and $\varphi_{q}=\varphi_{p}^{-1}$, $\frac{1}{p}+\frac{1}{q}=1$. $D^{\alpha}$ is the standard Caputo derivative and $f:[0,1]\times[0,\infty)\times(-\infty,\infty)\rightarrow[0,\infty)$ satisfies the Carath�odory type condition, $\alpha>2$ is real and $[\alpha]$ denotes the integer part of the real number $\alpha$, $\lambda>0$, $k_{i}, m_{i}$ ($i=0,1)$ are constants satisfying $0<k_{1}<k_{0}$ and $0<m_{1}<m_{0}$. They used the nonlinear alternative of
Leray-Schauder type and the fixed-point theorems in Banach space to investigate the existence of at least single, twin, triple, $n$ or $2n-1$ positive solutions.

In 2019, T. Xiaosong et al. \cite{Xiaosong} considered the following mixed fractional resonant boundary value
problem with $p(t)$-Laplacian operator
\begin{equation*}\label{}
 \begin{cases}
    ^{c}D^{\beta}\Big(\varphi_{p(t)}\big(D^{\alpha}u(t)\big)\Big)=f(t,u(t),D^{\alpha}u(t)),\  t \in [0,T],\\
    t^{1-\alpha}u(t)\big|_{t=0}=0,\ D^{\alpha}u(0)=D^{\alpha}u(T), \\
  \end{cases}
  \end{equation*}
where $0<\alpha,\beta\leq 1$, $1<\alpha+\beta\leq 2$, $^{c}D^{\beta}$ is Caputo fractional derivative and $D^{\alpha}$ is Riemann-Liouville fractional derivative, $\varphi_{p(t)}(.)$ is $p(t)$-Laplacian operator, $p(t)>1$, $p(t)\in C^{1}[0,T]$ with $p(0)=p(T)$, $f:[0,T]\times \mathbb{R}^{2}\rightarrow \mathbb{R}$. Under the appropriate conditions of the nonlinear term, the existence of solutions for the above mixed fractional resonant boundary value problem is obtained by using the continuation theorem of coincidence degree theory.

In 2019, Z. Li et al. \cite{Li} considered the existence of nontrivial solutions for a certain $p$-Laplacian fractional differential equation

\begin{equation*}\label{}
 \begin{cases}
     D^{\beta}\varphi_{p}\Big(\big( D^{\alpha}\big(p(t)u^{\prime}(t)\big)\big)\Big)+f(t,u(t))=0,\  t \in (0,1),\\
    au(0)-bp(0)u^{\prime}(0)=0,\ cu(1)+dp(1)u^{\prime}(1)=0,\  D^{\alpha}\big(p(t)u^{\prime}(t)\big)\big|_{t=0}=0,
  \end{cases}
    \end{equation*}
where $a$, $b$, $c$, $d$ are constants with $0<ad+bc+ac\int_{0}^{1}\frac{1}{p(s)}ds<\infty$, $p(.):[0,1]\rightarrow (0,\infty)$ is continuous, and  $\varphi_{p}$ is the $p$-Laplacian operator, $0<\alpha, \beta<1$, $D^{\alpha}$ is the standard Caputo derivative. Under the assumption that $f(t,u)$ is a continuous function, and by the use of some fixed point theorems in cones, they studied the existence and uniqueness results.

In 2017, T. Yuansheng et al. \cite{Yuansheng} considered a class of four-point boundary value problem of fractional differential equations with $p$-Laplacian operator
\begin{equation*}\label{}
 \begin{cases}
     D^{\gamma}\Big(\varphi_{p}\big( D^{\alpha}u(t)\big)\Big)=f(t,u(t)),\  t \in (0,1),\\
    u(0)=D^{\alpha}u(0)=0,\ D^{\beta}u(1)=\lambda u(\xi),\ D^{\alpha}u(1)=\mu D^{\alpha}u(\eta),
  \end{cases}
    \end{equation*}
where $\alpha,\beta,\gamma \in \mathbb{R}$, $1<\alpha, \gamma \leq2$, $\beta>0$, $1+\beta\leq \alpha$, and $\xi,\eta \in(0,1)$, $\lambda, \mu \in[0,\infty)$, $(1-\beta)\Gamma(\alpha)\geq \lambda \Gamma(\alpha-\beta)\xi^{\alpha-2}$, $1-\mu^{p-1}\eta^{\gamma-2}\geq0$ and $\varphi_{p}$ is the $p$-Laplacian operator, $D^{\nu}$ ($\nu \in\{\alpha,\beta,\gamma\}$) is the standard Riemann-Liouville differentiation and $f\in C([0,1]\times [0,\infty)\rightarrow[0,\infty))$. By the use of the Leggett-Williams fixed-point theorem, the multiplicity results of positive solution are obtained.

And in the same year, L. Xiping et al. \cite{Xiping}, studied the following four-point boundary value problem of fractional differential equation with mixed fractional derivatives and $p$-Laplacian operator
\begin{equation*}\label{}
 \begin{cases}
    D^{\alpha}\Big(\varphi_{p}\big(^{c}D^{\beta}u(t)\big)\Big)=f\big(t,u(t),^{c}D^{\beta}u(t)\big),\  t \in (0,1],\\
    ^{c}D^{\beta}u(0)= u^{\prime}(0)=0,\\
  u(1)=r_{1}u(\eta),\ ^{c}D^{\beta}u(1)=r_{2}\ ^{c}D^{\beta}u(\xi),
  \end{cases}
    \end{equation*}
where $1< \alpha,\beta\leq 2$, $r_{1},r_{2}\geq0$, $D^{\alpha}$ is  the  Riemann-Liouville  fractional  derivative  operator, and $^{c}D^{\beta}$ is  the  Caputo  fractional  derivative  operator, $p >1$, $\varphi_{p}$ is  the $p$-Laplacian  operator and $f\in C([0,1]\times[0,+\infty)\times(-\infty,0],[0,+\infty))$. By means a method of lower and upper solutions, the authors established some new results on the existence of positive solutions.

In 2018, W. Han et al. \cite{Han} investigated the multiple positive solutions for the following nonlinear fractional differential equations coupled with the $p$-Laplacian operator and infinite-point boundary value conditions
\begin{equation*}\label{}
 \begin{cases}
    -D^{\beta}\Big(\varphi_{p}\big( D^{\alpha}u(t)-\lambda u(t)\big)\Big)=f(t,u(t)),\  t \in (0,1],\\
    \lim_{t\rightarrow 0^{+}}t^{1-\alpha}u(t)=\sum_{i=1}^{\infty}\mu_{i}u(\xi_{i}),\\
 \lim_{t\rightarrow 0^{+}}t^{2-\beta}\Big(\varphi_{p}\big( D^{\alpha}u(t)-\lambda u(t)\big)\Big)=\varphi_{p}\big( D^{\alpha}u(1)-\lambda u(1)\big)=0,
  \end{cases}
    \end{equation*}
where $0<\alpha\leq1$, $1<\beta\leq2$, $\lambda<0$, $\mu_{i}\geq0$, $0<\xi_{i}<1$, $i\in\mathbb{N}^{+}$, $\sum_{i=1}^{\infty}\mu_{i}\xi_{i}^{\alpha-1}<1$, $f\in C([0,1]\times [0,\infty)\rightarrow[0,\infty))$, $\varphi_{p}$ is the $p$-Laplacian operator and $D^{\alpha}$, $D^{\beta}$ are the Riemann-Liouville
fractional derivatives. By means of the properties of Green's function and fixed point theorems,
they established the suitable criteria to guarantee the existence of positive solutions.

 In 2018, J. Tan, M. Li \cite{Tan}, studied the solutions for the following nonlinear fractional differential equations
with $p$-Laplacian operator nonlocal boundary value problem in a Banach space $E$
\begin{equation*}\label{}
 \begin{cases}
    -D^{\beta}\Big(\varphi_{p}\big(D^{\alpha}u\big)\Big)(t)=f(t,u(t)),\  t \in (0,1),\\
    u(0)=\theta,\ D^{\alpha}u(0)=\theta, \\
  D^{\gamma}u(1)=\sum_{i=1}^{m-2}\alpha_{i}D^{\gamma}u(\xi_{i}),
  \end{cases}
    \end{equation*}
where $D^{\alpha}, D^{\beta}, D^{\gamma}$ are the standard Riemann-Liouville fractional derivatives, $\theta$ is the zero element of $E$, $1<\alpha\leq2$, $0<\beta,\gamma\leq1$, $\alpha-\gamma\geq1$, $I=[0,1]$, $f:I\times E\rightarrow E$ is continuous, $\alpha_{i}\geq 0$ ($i=1,2,...,m-2$), $0<\xi_{1}<\xi_{2}<...<\xi_{m-2}<1$, $\sum_{i=1}^{m-2}\alpha_{i}\xi_{i}^{\alpha-\gamma-1}<1$, $\varphi_{p}$ is the $p$-Laplacian operator. By means of the technique of the properties of the Kuratowski noncompactness measure and the Sadovskii fixed point theorem, they established some new existence criteria.

For some other results on $p$-Laplacian fractional boundary value problems, we refer the reader to the papers
\cite{Wang,Perera,Perera2,Pucci,Chen,Hu2,Liu,Khan2,Cheng,Prasad}.

 Motivated by these works, in this paper, we are concerned with the following three-point boundary value problem of fractional differential equations with $p$-Laplacian operator
 \begin{equation}\label{eq1}
  \Big(\varphi_{p}\big(D^{\alpha}u(t)\big)\Big)^{\prime} +a(t)f(t,u(t)) = 0,\  t \in (0,1),
  \end{equation}
  \begin{equation}\label{eq2}
    D^{\alpha}u(0)= u^{\prime} (0) =u^{\prime \prime}(0) =0,\ u(1)+ u^{\prime }(1)=u^{\prime}(\eta),
  \end{equation}
  where $2<\alpha\leq3$, $\eta\in(0,1)$, $\varphi_{p}$ is the $p$-Laplacian operator, $D^{\alpha}$ is the standard Caputo derivative.
\newline
Throughout this paper, we assume the following conditions:
  \begin{itemize}
 \item[(H1)] $f \in C([0,1]\times [0,\infty) ,[0,\infty)),$ and $f(t,0)\centernot\equiv 0$ on $[0,1];$
\item[(H2)] $a \in C([0,1],[0,\infty))$ and $a(t)\centernot\equiv 0$ on any subinterval of $[0,1].$
\end{itemize}

 This paper is organized as follows. In section 2, we recall some special definitions, theorems and lemmas that will be used to prove our main results. In section 3, we discuss the existence and uniqueness of positive solution for  \eqref{eq1}-\eqref{eq2}. Finally, we give some examples to illustrate our results in section 4.

 \section{Preliminaries}
In this section, we introduce first the following preliminary facts that will be used throughout this article.
\\
At first, let $\mathcal{B}$ be the Banach space $C([0,1],\mathbb{R})$ when equipped with the usual supremum norm,
 \[\|u\|=\\sup_{t\in[0, 1]}|u(t)|.\]
 \begin{definition}
 Let $E$ be a real Banach space. A nonempty, closed, convex set $
 K\subset E$ is a cone if it satisfies the following two conditions:
 \begin{itemize}
 \item[(i)]
  $x\in K$, $\lambda \geq 0$ imply $\lambda x\in K$;
 \item[(ii)]
 $x\in K$, $-x\in K$ imply $x=0$.
 \end{itemize}
 \end{definition}

 \begin{definition}
 An operator $T:E\rightarrow E$ \ is completely continuous if it is continuous
 and maps bounded sets into relatively compact sets.
 \end{definition}
 \begin{definition}
 $ A $ function $ u(t)$ is called a positive solution of \eqref{eq1} and \eqref{eq2}
 if $ u \in C ([0,1]) $ and $ u(t) > 0 $ for all $ t \in(0,1).$
 \end{definition}

\begin{definition}
The Riemann-Liouville fractional integral of order $\alpha$ for a continuous function $f$ is defined as
\begin{equation*}
I^{\alpha}f(t)=\frac{1}{\Gamma(\alpha)}\int_{0}^{t}\frac{f(s)}{(t-s)^{1-\alpha}}ds, \ \alpha>0,
\end{equation*}
provided the integral exists, where $\Gamma(.)$ is the gamma function, which is defined by $\Gamma(x)=\int_{0}^{\infty}t^{x-1}e^{-t}dt$.
\end{definition}

\begin{definition}
For at least n-times continuously differentiable function $f:[0,\infty)\rightarrow \mathbb{R}$, the Caputo derivative of fractional order $\alpha$ is defined as
\begin{equation*}
^{c}D^{\alpha}f(t)=\frac{1}{\Gamma(n-\alpha)}\int_{0}^{t}\frac{f^{(n)}(s)}{(t-s)^{\alpha+1-n}}ds,\ n-1<\alpha<n,\ n=[\alpha]+1,
\end{equation*}
where $[\alpha]$ denotes the integer part of the real number $\alpha$.
\end{definition}

\begin{lemma}[\cite{Kilbas2006}] \label{lem 2.1}
For $\alpha> 0$, the general solution of the fractional differential equation $^{c}D^{\alpha}x(t)=0$ is
given by
\begin{equation*}
x(t)=c_{0}+c_{1}t+...+c_{n-1}t^{n-1},
\end{equation*}
where $c_{i}\in \mathbb{R}$, $i = 0, 1,...,n-1 \ (n = [\alpha] + 1)$.
\end{lemma}
According to Lemma \ref{lem 2.1}, it follows that
\begin{equation*}
{I^{\alpha}}\ {^{c}D^{\alpha}}x(t)=x(t)+c_{0}+c_{1}t+...+c_{n-1}t^{n-1},
\end{equation*}
for some $c_{i}\in \mathbb{R}$, $i = 0, 1,...,n-1 \ (n = [\alpha] + 1)$.

\begin{lemma}[\cite{Podlubny1999}, \cite{Kilbas2006}]\label{lem 2.2} If $\beta> \alpha>0$ and $x \in L_{1}[0,1]$, then
\item[(i)] ${^{c}D^{\alpha}}\ {I^{\beta}}x(t)={I^{\beta-\alpha}}x(t)$, holds almost everywhere on $[0,1]$ and it is valid at any point $t\in [0,1]$ if $x\in C[0,1]$;\  ${^{c}D^{\alpha}}\ {I^{\alpha}}x(t)=x(t) $, for all \ $t\in[0,1]$.
\item[(ii)] ${^{c}D^{\alpha}}t^{\lambda-1}=\frac{\Gamma(\lambda)}{\Gamma(\lambda-\alpha)}t^{\lambda-\alpha-1}$,
$\lambda>[\alpha]$ \ and \ ${^{c}D^{\alpha}}t^{\lambda-1}=0$, \  $\lambda<[\alpha]$.
\end{lemma}

\begin{definition}
Let $p > 0$, $q > 0$, the Euler beta function is defined by
\[\mathfrak{B}(p,q)=\int_{0}^{1}t^{p-1}(1-t)^{q-1}dt.\]
\end{definition}

The basic properties of the $\Gamma$ and $\mathfrak{B}$ functions which will be used in the following studies are listed below.
\begin{proposition}
Let $\alpha>0$, $p>0$, $q>0$ and $n$ a positive integer. Then
\[\Gamma(\alpha+1)=\alpha\Gamma(\alpha),\ \Gamma(n+\frac{1}{2})=\frac{\sqrt{\pi}(2n)!}{2^{2n}n!}, \Gamma(n+1)=n!,\]
\[\mathfrak{B}(p,q)=\frac{\Gamma(p)\Gamma(q)}{\Gamma(p+q)},\ \mathfrak{B}(p,q)=\mathfrak{B}(p,q+1)+\mathfrak{B}(p+1,q),\]
\[\mathfrak{B}(p,q)=\mathfrak{B}(q,p),\ \mathfrak{B}(p+1,q)=\mathfrak{B}(p,q)\frac{p}{p+q}.\]
In particular,
\[\Gamma(1)=1,\ \Gamma\bigg(\frac{1}{2}\bigg)=\sqrt{\pi},\ \mathfrak{B}(1,q)=\frac{1}{q}.\]
\end{proposition}

To prove our results, we need the following three well-known fixed point theorems.

 \begin{theorem}\label{T1} \cite{Kras}
 Let $E$ be a Banach space, and let $K\subset E$, be a cone. Assume that $%
 \Omega_{1}$ and $\Omega_{2}$ are bounded open subsets of $E$ with $0\in \Omega _{1}$,
 $\overline{{\Omega }}_{1}\subset \Omega_{2}$ and let
 \[
 A:K\cap  ( \overline{{\Omega}}_{2} \backslash \Omega_{1} )\rightarrow K
 \]
 be a completely continuous operator such that
 \begin{itemize}
 \item[(a)]
 $\left\Vert Au\right\Vert \leq \left\Vert u\right\Vert ,$ $u\in K\cap
 \partial
 \Omega _{1}$, and $\left\Vert Au\right\Vert \geq \left\Vert u\right\Vert ,$
 $u\in K\cap \partial \Omega_{2}$; or
 \item[(b)]
 $\left\Vert Au\right\Vert \geq \left\Vert u\right\Vert ,$ $u\in K\cap
 \partial
 \Omega_{1}$, and  $\left\Vert Au\right\Vert \leq \left\Vert u\right\Vert ,$
 $u\in K\cap \partial \Omega_{2}.$
 \end{itemize}
 Then $A$ has a fixed point in $K\cap  ( \overline{{\Omega }}_{2} \backslash \Omega_{1} )$.
 \end{theorem}

\begin{theorem}\label{T2} \cite{Granas2003}
Let $E$ be a Banach space, $E_{1}$ a closed, convex subset of $E$, $U$ an open subset of $E_{1}$                      and $0\in U$. Suppose that $A:\bar{U}\rightarrow E_{1}$ is a continuous, compact (that is $A(\bar{U})$ is a relatively
compact subset of $E_{1}$) map. Then either
\begin{itemize}
\item[(i)] $A$ has a fixed point in $\bar{U}$, or
\item[(ii)] There is a $u\in \partial {U}$ (the boundary of $U$ in $E_{1}$) and $\lambda \in (0,1)$ with $u=\lambda A(u).$
\end{itemize} \end{theorem}

\begin{theorem}\label{T3} (Banach's fixed point theorem)
 Let $(X,d)$ be a non-empty complete metric space with a contraction mapping $T:X\rightarrow X$. Then $T$ admits a unique fixed-point $u^{*}$ in $X$ (i.e. $T(u^{*}) = u^{*}$).
\end{theorem}

 Consider the fractional boundary value problem
 \begin{equation}\label{eq3}
 \Big(\varphi_{p}\big(D^{\alpha}u(t)\big)\Big)^{\prime} +h(t)= 0,\  t \in (0,1),
 \end{equation}
 \begin{equation}\label{eq4}
  D^{\alpha}u(0)= u^{\prime} (0) =u^{\prime \prime}(0) =0,\ u(1)+ u^{\prime }(1)=u^{\prime}(\eta),
 \end{equation}
where $\alpha\in(2,3]$ and $h\in C[0,1]$.

\begin{lemma}\label{lem2.3}
Suppose that  $h\in C([0,1],[0,\infty))$, then the boundary value problem \eqref{eq3}-\eqref{eq4} has a unique solution which can be expressed by
$$ u(t) = \int_{0}^{1}  \mathcal{K}(t,s)\varphi_{q}\bigg(\int_{0}^{s}h(\tau)d\tau\bigg) ds ,$$
where
 \begin{equation}\label{eq5}
\mathcal{K}(t,s)=\mathcal{G}(t,s)+\mathcal{H}(\eta,s),
 \end{equation}

\begin{eqnarray}\label{eq6}
\mathcal{G}(t,s)& =& \begin{cases}\frac{(1-s)^{\alpha-1}-(t-s)^{\alpha-1}}{\Gamma(\alpha)}, & 0 \leq s \leq t \leq 1, \\
\frac{(1-s)^{\alpha-1}}{\Gamma(\alpha)} , & 0 \leq t \leq s \leq 1,
\end{cases}\notag\\
\mathcal{H}(t,s)& =& \begin{cases}\frac{(1-s)^{\alpha-2}-(t-s)^{\alpha-2}}{\Gamma(\alpha-1)}, & 0 \leq s \leq t \leq 1, \\
\frac{(1-s)^{\alpha-2}}{\Gamma(\alpha-1)} , & 0 \leq t \leq s \leq 1 .
\end{cases}
\end{eqnarray}
\end{lemma}

\begin{proof}
By integrating the equation \eqref{eq3}, it follows that
$$ \varphi_{p}\big(D^{\alpha}u(t)\big)-\varphi_{p}\big(D^{\alpha}u(0)\big)= - \int_{0}^{t} h(s) ds,$$
and so,
$$ D^{\alpha}u(t)=-\varphi_{q}\bigg(\int_{0}^{t} h(s) ds \bigg). $$
From Lemma \ref{lem 2.1} , we get
\begin{equation*} \label{}
u(t) = - I^{\alpha}\varphi_{q}\bigg(\int_{0}^{t} h(s) ds\bigg)+ C_{0}+C_{1}t+ C_{2}t^{2}.
\end{equation*}
where $C_{0}, C_{1}, C_{2} \in \mathbb{R}$ are constants.
Using the boundary conditions \eqref{eq4}, we have \\
\[C_{1} =C_{2}= 0.\]
So,
\begin{eqnarray}\label{eq7}
u(t) & =& - \int_{0}^{t}\frac{(t-s)^{\alpha-1}}{\Gamma(\alpha)}\varphi_{q}\bigg(\int_{0}^{s} h(\tau) d\tau \bigg)ds+ C_{0},\\
u^{\prime}(t)& =& - \int_{0}^{t}\frac{(t-s)^{\alpha-2}}{\Gamma(\alpha-1)}\varphi_{q}\bigg(\int_{0}^{s} h(\tau) d\tau \bigg)ds \notag.
\end{eqnarray}
By the boundary condition $u(1)+ u^{\prime }(1)=u^{\prime}(\eta)$, we have
\begin{equation*}
\begin{split}
C_{0}= & \int_{0}^{1}\frac{(1-s)^{\alpha-1}}{\Gamma(\alpha)}\varphi_{q}\bigg(\int_{0}^{s} h(\tau) d\tau \bigg)ds
+\int_{0}^{1}\frac{(1-s)^{\alpha-2}}{\Gamma(\alpha-1)}\varphi_{q}\bigg(\int_{0}^{s} h(\tau) d\tau \bigg)ds\\
 &- \int_{0}^{\eta}\frac{(\eta-s)^{\alpha-2}}{\Gamma(\alpha-1)}\varphi_{q}\bigg(\int_{0}^{s} h(\tau) d\tau \bigg)ds.
\end{split}
\end{equation*}
Inserting $C_{0}$ into \eqref{eq7}, we get

\begin{eqnarray*}
u(t)& =& - \int_{0}^{t}\frac{(t-s)^{\alpha-1}}{\Gamma(\alpha)}\varphi_{q}\bigg(\int_{0}^{s} h(\tau) d\tau \bigg)ds
+\int_{0}^{1}\frac{(1-s)^{\alpha-1}}{\Gamma(\alpha)}\varphi_{q}\bigg(\int_{0}^{s} h(\tau) d\tau \bigg)ds\\
 && +  \int_{0}^{1}\frac{(1-s)^{\alpha-2}}{\Gamma(\alpha-1)}\varphi_{q}\bigg(\int_{0}^{s} h(\tau) d\tau \bigg)ds
 -\int_{0}^{\eta}\frac{(\eta-s)^{\alpha-2}}{\Gamma(\alpha-1)}\varphi_{q}\bigg(\int_{0}^{s} h(\tau) d\tau \bigg)ds\\
&=& - \int_{0}^{t}\frac{(t-s)^{\alpha-1}}{\Gamma(\alpha)}\varphi_{q}\bigg(\int_{0}^{s} h(\tau) d\tau \bigg)ds
+ \int_{0}^{t}\frac{(1-s)^{\alpha-1}}{\Gamma(\alpha)}\varphi_{q}\bigg(\int_{0}^{s} h(\tau) d\tau \bigg)ds\\
 && +\int_{t}^{1}\frac{(1-s)^{\alpha-1}}{\Gamma(\alpha)}\varphi_{q}\bigg(\int_{0}^{s} h(\tau) d\tau \bigg)ds
 +\int_{0}^{\eta}\frac{(1-s)^{\alpha-2}}{\Gamma(\alpha-1)}\varphi_{q}\bigg(\int_{0}^{s} h(\tau) d\tau \bigg)ds\\
  &&+\int_{\eta}^{1}\frac{(1-s)^{\alpha-2}}{\Gamma(\alpha-1)}\varphi_{q}\bigg(\int_{0}^{s} h(\tau) d\tau \bigg)ds
  -\int_{0}^{\eta}\frac{(\eta-s)^{\alpha-2}}{\Gamma(\alpha-1)}\varphi_{q}\bigg(\int_{0}^{s} h(\tau) d\tau \bigg)ds\\
 & =&  \int_{0}^{t}\frac{(1-s)^{\alpha-1}-(t-s)^{\alpha-1}}{\Gamma(\alpha)}\varphi_{q}\bigg(\int_{0}^{s} h(\tau) d\tau \bigg)ds\\
 &&+\int_{t}^{1}\frac{(1-s)^{\alpha-1}}{\Gamma(\alpha)}\varphi_{q}\bigg(\int_{0}^{s} h(\tau) d\tau \bigg)ds\\
  && +  \int_{0}^{\eta}\frac{(1-s)^{\alpha-2}-(\eta-s)^{\alpha-2}}{\Gamma(\alpha-1)}\varphi_{q}\bigg(\int_{0}^{s} h(\tau) d\tau \bigg)ds\\
&& +\int_{\eta}^{1}\frac{(1-s)^{\alpha-2}}{\Gamma(\alpha-1)}\varphi_{q}\bigg(\int_{0}^{s} h(\tau) d\tau \bigg)ds\\
 & =&\int_{0}^{1} \bigg( \mathcal{G}(t,s) + \mathcal{H}(\eta,s)\bigg) \varphi_{q}\bigg(\int_{0}^{s} h(\tau) d\tau \bigg)ds \\
 & =&\int_{0}^{1}  \mathcal{K}(t,s)\varphi_{q}\Big(\int_{0}^{s}h(\tau)d\tau\Big) ds.
\end{eqnarray*}
\end{proof}

\begin{lemma}\label{lem2.4} The Green's functions $\mathcal{G}(t,s)$ and $\mathcal{H}(t,s)$ defined by \eqref{eq6} are continuous on $[0,1]\times[0,1]$ and satisfy the following properties
\begin{itemize}
\item[(i)] $\mathcal{G}(t,s)\geq 0$, $\mathcal{H}(t,s) \geq 0 $, for all\ $t, s \in [0,1];$
\item[(ii)] $\big(1-t^{\alpha-1}\big)\mathcal{G}(s,s) \leq \mathcal{G}(t,s)\leq \mathcal{G}(s,s),$  \ for all \  $(t,s) \in [0,1] \times [0,1];$
\item[(iii)]   $\big(1-t^{\alpha-2}\big)\mathcal{H}(s,s) \leq \mathcal{H}(t,s)\leq \mathcal{H}(s,s),$  \ for all \  $(t,s) \in [0,1] \times [0,1].$
 \end{itemize}
 where $\mathcal{G}(s,s)=\frac{(1-s)^{\alpha-1}}{\Gamma(\alpha)}$, $\mathcal{H}(s,s)=\frac{(1-s)^{\alpha-2}}{\Gamma(\alpha-1)}.$
\end{lemma}
\begin{proof}
$\mathcal{G}$ and $\mathcal{H}$ are continuous by definition.\\
\rm{(i)} If $0\leq s \leq t$, then
\begin{eqnarray*}
\mathcal{G}(t,s)& =&\frac{(1-s)^{\alpha-1}-(t-s)^{\alpha-1}}{\Gamma(\alpha)}\\
& \geq& \frac{(1-s)^{\alpha-1}-(1-s)^{\alpha-1}}{\Gamma(\alpha)}\\
&=&0.
\end{eqnarray*}
For $s\geq t$, we have $\mathcal{G}(t,s)=\frac{(1-s)^{\alpha-1}}{\Gamma(\alpha)}\geq 0.$\\
So,
$$\mathcal{G}(t,s)\geq 0,\ \text{for all}\ t, s \in [0,1].$$
By a similar argument, we show that $\mathcal{H}(t,s)\geq 0,\ \text{for all}\ t, s \in [0,1].$\\
\rm{(ii)} If $0<s\leq t$, then
\begin{eqnarray*}
\mathcal{G}(t,s)& =&\frac{(1-s)^{\alpha-1}-(t-s)^{\alpha-1}}{\Gamma(\alpha)}\\
& \geq& \frac{(1-s)^{\alpha-1}-t^{\alpha-1}(1-\frac{s}{t})^{\alpha-1}}{\Gamma(\alpha)}\\
& \geq& \frac{(1-s)^{\alpha-1}-t^{\alpha-1}(1-s)^{\alpha-1}}{\Gamma(\alpha)}\\
&=&\big(1-t^{\alpha-1}\big)\frac{(1-s)^{\alpha-1}}{\Gamma(\alpha)}.
\end{eqnarray*}
On the other hand, we have $\mathcal{G}(t,s)\leq \frac{(1-s)^{\alpha-1}}{\Gamma(\alpha)}.$\\
Notice that
\[\mathcal{G}(s,s)=\frac{(1-s)^{\alpha-1}}{\Gamma(\alpha)}.\]
Thus,
\[\big(1-t^{\alpha-1}\big)\mathcal{G}(s,s) \leq \mathcal{G}(t,s)\leq \mathcal{G}(s,s),\ \text{for all}\ t, s \in [0,1].\]
Now, if $s=0$, then
\begin{eqnarray*}
\mathcal{G}(t,0)& =&\frac{1-t^{\alpha-1}}{\Gamma(\alpha)}\\
&\leq& \frac{1}{\Gamma(\alpha)}\\
& =& \mathcal{G}(0,0).
\end{eqnarray*}
On the other hand,
\[\mathcal{G}(t,0)=\frac{1-t^{\alpha-1}}{\Gamma(\alpha)}\geq \big(1-t^{\alpha-1}\big)\mathcal{G}(0,0).\]
Therefore, item \rm{(ii)} holds.\\
\rm{(iii)} It follows directly from \rm{(ii)}.
\end{proof}

\begin{lemma}\label{lem2.5}
Let $\rho\in (0,1)$ be fixed. $\mathcal{K}(t,s)$ defined by \eqref{eq5} satisfies the following properties
\begin{itemize}
\item[(i)] $ \mathcal{K}(t,s) \geq 0 $, for all\ $t, s \in [0,1].$
\item[(ii)] $\big(1-\eta^{\alpha-2}\big)\big(1-t^{\alpha-1}\big)\Phi(s)\leq \mathcal{K}(t,s)\leq \Phi(s),$  \ for all \  $(t,s) \in [0,1] \times [0,1].$
\item[(iii)] $\big(1-\eta^{\alpha-2}\big)\big(1-\rho^{\alpha-1}\big)\Phi(s)\leq \mathcal{K}(t,s)\leq \Phi(s),$  \ for all \  $(t,s) \in [0,\rho] \times [0,1],$
\end{itemize}
where $\Phi(s)=\frac{(\alpha-s)(1-s)^{\alpha-2}}{\Gamma(\alpha)}.$
\end{lemma}
\begin{proof}
Notice that \rm{(i)} holds trivially. Next, we show \rm{(ii)} holds.
First, Notice that
\begin{eqnarray*}
\mathcal{G}(s,s)+\mathcal{H}(s,s)& =& \frac{(1-s)^{\alpha-1}}{\Gamma(\alpha)}+\frac{(1-s)^{\alpha-2}}{\Gamma(\alpha-1)}\\
&=& \Phi(s).
\end{eqnarray*}

From Lemma \ref{lem2.4} and \eqref{eq5}, we have
\begin{eqnarray*}
\mathcal{K}(t,s)& =&\mathcal{G}(t,s)+\mathcal{H}(\eta,s)\\
& \leq& \mathcal{G}(s,s)+\mathcal{H}(s,s)\\
& =& \Phi(s).
\end{eqnarray*}
On the other hand, from Lemma \ref{lem2.4}, we get
\begin{eqnarray*}
\mathcal{K}(t,s)& =&\mathcal{G}(t,s)+\mathcal{H}(\eta,s)\\
& \geq& \big(1-t^{\alpha-1}\big)\mathcal{G}(s,s)+\big(1-{\eta}^{\alpha-2}\big)\mathcal{H}(s,s)\\
& \geq& \big(1-t^{\alpha-1}\big)\big(1-{\eta}^{\alpha-2}\big)\big[\mathcal{G}(s,s)+\mathcal{H}(s,s)\big]\\
& =& \big(1-t^{\alpha-1}\big)\big(1-{\eta}^{\alpha-2}\big)\Phi(s).
\end{eqnarray*}
Therefore, \rm{(ii)} holds.\\
\rm{(iii)} It follows directly from \rm{(ii)}.
\end{proof}

 \begin{lemma}\label{lem2.6}
 Let $\rho\in (0,1)$ be fixed. If $h\in C ([0,1],[0,\infty )) $, then the unique solution of the fractional boundary value problem \eqref{eq3}-\eqref{eq4} is nonnegative and satisfies
 \[ \min_{t\in [0,\rho]} u(t) \geq \gamma\|u\|,\]
 \end{lemma}
 where $\gamma=\big(1-\eta^{\alpha-2}\big)\big(1-\rho^{\alpha-1}\big).$
 \begin{proof}
 The positiveness of $ u(t) $ follows immediately from Lemma \ref{lem2.3} and Lemma \ref{lem2.5}.\\
For all $ t \in[0,1] $, we have
\begin{equation*}
 \begin{split}
  u(t)&=\int_{0}^{1}\mathcal{K}(t,s)\varphi_{q}\bigg(\int_{0}^{s}h(\tau)d\tau\bigg) ds \\
  &\leq \int_{0}^{1}\Phi(s)\varphi_{q}\bigg(\int_{0}^{s}h(\tau)d\tau\bigg) ds\\
   \end{split}
 \end{equation*}
 Then
 \begin{equation}\label{eq8}
 \|u\| \leq  \int_{0}^{1}\Phi(s)\varphi_{q}\bigg(\int_{0}^{s}h(\tau)d\tau\bigg) ds.
 \end{equation}

 On the other hand, Lemma \ref{lem2.3}, Lemma \ref{lem2.5} and  \eqref{eq8} imply that, for any $ t\in [0,\rho] $, we have
\begin{equation}\label{eq9}
\begin{split}
 u(t)&= \int_{0}^{1}\mathcal{K}(t,s)\varphi_{q}\bigg(\int_{0}^{s}h(\tau)d\tau\bigg) ds\\
 &\geq \big(1-\eta^{\alpha-2}\big)\big(1-\rho^{\alpha-1}\big)\int_{0}^{1}\Phi(s)\varphi_{q}\bigg(\int_{0}^{s}h(\tau)d\tau\bigg) ds\\
  &= \gamma \int_{0}^{1}\Phi(s)\varphi_{q}\bigg(\int_{0}^{s}h(\tau)d\tau\bigg) ds\\
  &\geq \gamma\|u\|.
\end{split}
\end{equation}
 Therefore,
  \[ \min_{t\in [0,\rho]} u(t) \geq \gamma\|u\|.\]
 \end{proof}

 Let $\rho\in (0,1)$ be fixed. Introduce the cone that we shall use in the sequel.

 $$\mathcal{P}= \left\lbrace u \in \mathcal{B}: \ u(t) \geq 0,\ t\in[0,1], \min_{t\in [0,\rho]} u(t) \geq \gamma\|u\| \right\rbrace,$$
 and define the operator $ \mathcal{A} : \mathcal{P} \rightarrow \mathcal{B} $ by
 \begin{equation}\label{eq10}
\mathcal{A}u(t)= \int_{0}^{1}\mathcal{K}(t,s)\varphi_{q}\bigg(\int_{0}^{s}a(\tau)f(\tau,u(\tau))d\tau\bigg) ds ,
 \end{equation}
 where $\mathcal{K}(t,s) $ is defined by \eqref{eq5}.
 \begin{remark}
 By Lemma \ref{lem2.3}, the fixed points of the operator $ \mathcal{A}$ in $ \mathcal{P} $ are the nonnegative solutions of the boundary value problem \eqref{eq1}-\eqref{eq2}.
 \end{remark}
The following properties of the $p$-Laplacian operator will play an important role in the rest of the paper.
 \begin{lemma} [\cite{Khan},\ Lemma 1.3]\label{lem2.7}
 Let $\varphi_{p}$ be a $p$-Laplacian operator. Then
 \begin{itemize}
\item[(i)] If $1<p\leq2$, $xy>0$, and $|x|, |y|\geq m>0$, then
\[|\varphi_{p}(x)-\varphi_{p}(y)|\leq (p-1)m^{p-2}|x-y|.\]
\item[(ii)] If $p>2$, $|x|, |y|\leq M$, then
\[|\varphi_{p}(x)-\varphi_{p}(y)|\leq (p-1)M^{p-2}|x-y|.\]
  \end{itemize}

 \end{lemma}
 \begin{lemma} \label{lem2.8}
 The operator $ \mathcal{A} $ defined in \eqref{eq10} is completely continuous and satisfies $ \mathcal{A}\mathcal{P} \subset \mathcal{P}.$
 \end{lemma}
 \begin{proof}
 From Lemma \ref{lem2.6} and under assumption \rm{(H1)}, it follows that $ \mathcal{A}\mathcal{P} \subset \mathcal{P}$. In view of the assumption of nonnegativeness and continuity of $f(t,u(t))$, $\mathcal{K}(t,s)$ and Lebesgue's dominated convergence theorem, we conclude that $ \mathcal{A} : \mathcal{P} \rightarrow \mathcal{P}$.
 Let $D$ be an arbitrary bounded set in $\mathcal{P}$. Then, there exists $M>0$ such that $D\subset\{u\in \mathcal{\mathcal{P}}: \|u\|<M\}.$
 Set $$L=\max\{f(t,u) / \ t\in [0,1], u\in{D}\}.$$
From Lemmas \ref{lem2.3} and \ref{lem2.5}, for any $u\in D$, we have
\begin{eqnarray*}
\mathcal{A}u(t)& =&\int_{0}^{1}\mathcal{K}(t,s)\varphi_{q}\bigg(\int_{0}^{s}a(\tau)f(\tau,u(\tau))d\tau\bigg) ds\\
&\leq& \int_{0}^{1}\mathcal{K}(t,s)\varphi_{q}\bigg(\int_{0}^{1}a(\tau)Ld\tau\bigg) ds\\
&\leq& L^{q-1}\varphi_{q}\bigg(\int_{0}^{1}a(\tau)d\tau\bigg)\int_{0}^{1}\Phi(s)ds.
\end{eqnarray*}
Thus,
\[\|\mathcal{A}u\|\leq L^{q-1}\varphi_{q}\bigg(\int_{0}^{1}a(\tau)d\tau\bigg)\int_{0}^{1}\Phi(s)ds.\]
Hence, $\mathcal{A}(D)$ is uniformly bounded.
On the other hand, let $u\in D$, $t_{1}, t_{2} \in [0,1]$ with $t_{1}<t_{2}.$
Then, from Lemmas \ref{lem2.3} and \ref{lem2.4}, we have

\begin{eqnarray*}
|\mathcal{A}u(t_{1})-\mathcal{A}u(t_{2})|&=&\Bigg|\int_{0}^{1}\mathcal{K}(t_{1},s)\varphi_{q}\bigg(\int_{0}^{s}a(\tau)f(\tau,u(\tau))d\tau\bigg) ds\\
&&-\int_{0}^{1}\mathcal{K}(t_{2},s)\varphi_{q}\bigg(\int_{0}^{s}a(\tau)f(\tau,u(\tau))d\tau\bigg) ds\Bigg|\\
&=& \Bigg|\int_{0}^{1}\Big[\mathcal{K}(t_{1},s)-\mathcal{K}(t_{2},s)\Big]\varphi_{q}\bigg(\int_{0}^{s}a(\tau)f(\tau,u(\tau))d\tau\bigg) ds\Bigg|\\
&=&\Bigg|\int_{0}^{1}\Big[\mathcal{G}(t_{1},s)-\mathcal{G}(t_{2},s)\Big]\varphi_{q}\bigg(\int_{0}^{s}a(\tau)f(\tau,u(\tau))d\tau\bigg) ds\Bigg|\\
&\leq& L^{q-1}\int_{0}^{1}\big|\mathcal{G}(t_{1},s)-\mathcal{G}(t_{2},s)\big|\varphi_{q}\bigg(\int_{0}^{s}a(\tau)d\tau\bigg) ds.\\
\end{eqnarray*}
The continuity of $\mathcal{G}$ implies that the right-side of the above inequality tends
to zero if $t_{2}\rightarrow t_{1}$. That is to say, $\mathcal{A}(D)$ is equicontinuous. Thus, the Arzela-Ascoli theorem implies that $\mathcal{A}:\mathcal{P}\rightarrow \mathcal{P}$ is completely continuous.
 \end{proof}

\section{Existence of positive solutions}
Set
$$\Lambda_{1}=\Bigg(\varphi_{q}\bigg(\int_{0}^{1}a(\tau)d\tau\bigg) \int_{0}^{1}\Phi(s)ds\Bigg)^{-1}, \ \Lambda_{2}=\Bigg(\gamma \int_{0}^{\rho}\Phi(s)\varphi_{q}\bigg(\int_{0}^{s}a(\tau)d\tau\bigg)ds\Bigg)^{-1}.$$

Then we find that $0<\Lambda_{1}<\Lambda_{2}.$ In fact,
\begin{eqnarray*}
\Lambda_{2}^{-1}& =&\gamma \int_{0}^{\rho}\Phi(s)\varphi_{q}\bigg(\int_{0}^{s}a(\tau)d\tau\bigg)ds\\
&<& \int_{0}^{\rho}\Phi(s)\varphi_{q}\bigg(\int_{0}^{s}a(\tau)d\tau\bigg)ds\\
&\leq& \varphi_{q}\bigg(\int_{0}^{1}a(\tau)d\tau\bigg) \int_{0}^{1}\Phi(s)ds\\
& =& \Lambda_{1}^{-1}.
\end{eqnarray*}
\begin{theorem}\label{th3.1}
Assume that \rm{(H1)}-\rm{(H2)} hold. If there exist constants $\rho_{1}>0$, $\rho_{2}>0$, $M_{1}\in (0,\Lambda_{1}]$, and $M_{2}\in [\Lambda_{2},\infty)$, where $\rho_{1}<\rho_{2}$ and $M_{2}\rho_{1}<M_{1}\rho_{2}$, such that $f$ satisfies
\begin{itemize}
\item[(i)] $ f(t,u)\leq \varphi_{p}(M_{1}\rho_{2})$\ for all $u\in[0,\rho_{2}]$  and $t \in [0,1]$, and
\item[(ii)]  $ f(t,u)\geq\varphi_{p}(M_{2}\rho_{1})$\ for all $u\in[\gamma \rho_{1},\rho_{1}]$  and $t\in [0,\rho]$,
\end{itemize}
then the problem \eqref{eq1}-\eqref{eq2} has at least one positive solution $u \in \mathcal{P}$
satisfying $\rho_{1}< \|u\|<\rho_{2}.$
\end{theorem}

\begin{proof}
Define the open set
\[\Omega_{\rho_{2}}=\{u\in\mathcal{B}: \|u\|<\rho_{2}\}.\]
Let $u\in \mathcal{P}\cap \partial \Omega_{\rho_{2}}.$ Then, from assumption \rm{(i)}
and Lemma \ref{lem2.5}, we have
\begin{eqnarray*}
\mathcal{A}u(t)& =&\int_{0}^{1}\mathcal{K}(t,s)\varphi_{q}\bigg(\int_{0}^{s}a(\tau)f(\tau,u(\tau))d\tau\bigg) ds\\
&\leq& \int_{0}^{1}\mathcal{K}(t,s)\varphi_{q}\bigg(\int_{0}^{s}a(\tau)\varphi_{p}(M_{1}\rho_{2})d\tau\bigg) ds\\
&\leq& M_{1}\rho_{2} \int_{0}^{1}\mathcal{K}(t,s)\varphi_{q}\bigg(\int_{0}^{1}a(\tau)d\tau\bigg) ds \\
&\leq& M_{1}\rho_{2} \varphi_{q}\bigg(\int_{0}^{1}a(\tau)d\tau\bigg) \int_{0}^{1}\Phi(s)ds\\
&\leq& \Lambda_{1} \Lambda_{1}^{-1} \rho_{2}\\
&=& \rho_{2},
\end{eqnarray*}
which implies that
\begin{equation} \label{eq11}
\|\mathcal{A}u\|\leq \|u\| \ \text{for all}\ u\in \mathcal{P}\cap \partial \Omega_{\rho_{2}}.
\end{equation}
Next, define the open set $\Omega_{\rho_{1}}=\{u\in\mathcal{B}: \|u\|<\rho_{1}\}.$\\
For any $u\in \mathcal{P}\cap \partial \Omega_{\rho_{1}}$, by using \rm{(H1)}-\rm{(H2)}, assumption \rm{(ii)} and Lemma \ref{lem2.5}, for $t \in [0,\rho]$, we then get
\begin{eqnarray*}
\mathcal{A}u(t)& =&\int_{0}^{1}\mathcal{K}(t,s)\varphi_{q}\bigg(\int_{0}^{s}a(\tau)f(\tau,u(\tau))d\tau\bigg) ds\\
&\geq& \int_{0}^{\rho}\mathcal{K}(t,s)\varphi_{q}\bigg(\int_{0}^{s}a(\tau)f(\tau,u(\tau))d\tau\bigg) ds\\
&\geq&\gamma \int_{0}^{\rho}\Phi(s)\varphi_{q}\bigg(\int_{0}^{s}a(\tau)f(\tau,u(\tau))d\tau\bigg) ds \\
&\geq& M_{2}\rho_{1} \gamma \int_{0}^{\rho}\Phi(s)\varphi_{q}\bigg(\int_{0}^{s}a(\tau)d\tau\bigg) ds\\
&\geq& \rho_{1} \Lambda_{2} \gamma \int_{0}^{\rho}\Phi(s)\varphi_{q}\bigg(\int_{0}^{s}a(\tau)d\tau\bigg)ds\\
&=& \rho_{1} \Lambda_{2} \Lambda_{2}^{-1}\\
&=& \rho_{1},
\end{eqnarray*}
which implies that
\begin{equation} \label{eq12}
\|\mathcal{A}u\| \geq \|u\| \ \text{for all}\ u\in \mathcal{P}\cap \partial \Omega_{\rho_{1}}.
\end{equation}
Therefore by \rm{(b)} in Theorem \ref{T1}, $\mathcal{A}$ has at least one fixed point in $\mathcal{P} \cap (\bar{\Omega}_{\rho_{2}}\setminus \Omega_{\rho_{1}})$. So there exists at least one solution of \eqref{eq1}-\eqref{eq2} with $\rho_{1}< \|u\|<\rho_{2}.$
\end{proof}

By a closely similar way, we can obtain the following result.

\begin{theorem}\label{th3.2}
Assume that \rm{(H1)}-\rm{(H2)} hold. If there exist constants $\rho_{1}>0$, $\rho_{2}>0$, $M_{1}\in (0,\Lambda_{1}]$, and $M_{2}\in [\Lambda_{2},\infty)$, where $\gamma \rho_{2}<\rho_{1}<\rho_{2}$, and $M_{1}\rho_{1}> M_{2}\rho_{2}$, such that $f$ satisfies
\begin{itemize}
\item[(i)] $ f(t,u)\geq \varphi_{p}(M_{2}\rho_{2})$\ for all $u\in [\gamma \rho_{2},\rho_{2}]$  and $t \in [0,\rho]$, and
\item[(ii)]  $ f(t,u) \leq \varphi_{p}(M_{1}\rho_{1})$\ for all $u\in [0,\rho_{1}]$  and $t\in [0,1]$,
\end{itemize}
then the problem \eqref{eq1}-\eqref{eq2} has at least one positive solution $u \in \mathcal{P}$
satisfying $\rho_{1}< \|u\|<\rho_{2}.$
\end{theorem}

\begin{theorem}\label{th3.3}
Assume that \rm{(H1)}-\rm{(H2)} hold. Further, assume that there exists a constant $\nu>0$ such that
\begin{equation} \label{eq13}
\nu> L^{q-1}\varphi_{q}\bigg(\int_{0}^{1}a(\tau)d\tau\bigg) \int_{0}^{1}\Phi(s)ds,
\end{equation}
where $L=\max\{f(t,u) / \ (t,u)\in[0,1]\times [0,\nu]\}.$\\
Then the fractional boundary value problem \eqref{eq1}-\eqref{eq2} has at least one positive
solution.
\end{theorem}

\begin{proof}
Let
\[\mathcal{U}=\{u\in \mathcal{P}:\|u\|<\nu\}.\]
By virtue of Lemma \ref{lem2.8}, the operator $\mathcal{A}:\overline{\mathcal{U}}\rightarrow \mathcal{P}$ is completely continuous. Assume that there exist $u\in \overline{\mathcal{U}}$ and $\lambda\in(0,1)$ such that $u=\lambda \mathcal{A}u.$ Then we have
\begin{eqnarray*}
|u(t)|=|\lambda(\mathcal{A}u)(t)|& =&\bigg|\lambda\int_{0}^{1}\mathcal{K}(t,s)\varphi_{q}\bigg(\int_{0}^{s}a(\tau)f(\tau,u(\tau))d\tau\bigg) ds\bigg|\\
&\leq& \int_{0}^{1}\mathcal{K}(t,s)\varphi_{q}\bigg(\int_{0}^{s}a(\tau)L d\tau\bigg) ds\\
&\leq& L^{q-1}\varphi_{q}\bigg(\int_{0}^{1}a(\tau)d\tau\bigg)\int_{0}^{1}\Phi(s) ds.
\end{eqnarray*}
So,
\[\|u\|\leq L^{q-1}\varphi_{q}\bigg(\int_{0}^{1}a(\tau)d\tau\bigg)\int_{0}^{1}\Phi(s) ds.\]
Thus from \eqref{eq13}, we have that $\|u\|<\nu$, which means that $u\centernot\in \partial\mathcal{U}.$ Hence, it follows that there is no $u\in\partial\mathcal{U}$ such that $u=\lambda\mathcal{A}u$ for some $\lambda \in(0,1).$ Therefore by Theorem \ref{T2}, we conclude that the fractional boundary value problem \eqref{eq1}-\eqref{eq2} has at least one positive
solution.
\end{proof}

\begin{theorem}\label{th3.5}
Assume that \rm{(H1)}-\rm{(H2)} hold, and $1<p<2$. In addition, we assume that the following assumptions hold:
\begin{itemize}
\item[(C1)] There exists a nonnegative function $k\in C[0,1]$ such that
\begin{equation} \label{eq14}
f(t,u)\leq k(t),\ \text{for any}\ (t,u)\in[0,1]\times [0,\infty).
\end{equation}
\item[(C2)]There exists a constant $L$ with $0<L<\frac{\Gamma(\alpha+1)}{(\alpha+1)(q-1)}\Big(\int_{0}^{1}a(t)dt\Big)^{-1}\Big(\int_{0}^{1}a(t)k(t)dt\Big)^{2-q}$ such that
\begin{equation}\label{eq15}
|f(t,u)-f(t,v)|\leq L |u-v|, \ \text{for any}\ t\in[0,1] \ \text{and}\ u,v\in[0,\infty).
\end{equation}
\end{itemize}
Then the boundary value problem \eqref{eq1} and \eqref{eq2} has a unique solution.
\end{theorem}

\begin{proof}
By \eqref{eq14}, for $t\in[0,1]$, we get
\begin{equation*}
\begin{split}
\int_{0}^{t} a(s)f(s,u(s))ds
&\leq \int_{0}^{1}a(s)f(s,u(s))ds\\
&\leq \int_{0}^{1}a(s)k(s)ds=M.
\end{split}
\end{equation*}
From \rm{(ii)} in Lemma \ref{lem2.7} and \eqref{eq15}, for any $u,v\in\mathcal{B}$, we have
\begin{eqnarray*}
|\mathcal{A}u(t)-\mathcal{A}v(t)|&=&\Bigg|\int_{0}^{1}\mathcal{K}(t,s)\varphi_{q}\bigg(\int_{0}^{s}a(\tau)f(\tau,u(\tau))d\tau\bigg) ds\\
&&-\int_{0}^{1}\mathcal{K}(t,s)\varphi_{q}\bigg(\int_{0}^{s}a(\tau)f(\tau,v(\tau))d\tau\bigg) ds\Bigg|\\
&=&\Bigg|\int_{0}^{1}\mathcal{K}(t,s)\Bigg(\varphi_{q}\bigg(\int_{0}^{s}a(\tau)f(\tau,u(\tau))d\tau\bigg)\\
&&-\varphi_{q}\bigg(\int_{0}^{s}a(\tau)f(\tau,v(\tau))d\tau\bigg)\Bigg)ds\Bigg|\\
&\leq& \int_{0}^{1}\mathcal{K}(t,s)\Bigg|\varphi_{q}\bigg(\int_{0}^{s}a(\tau)f(\tau,u(\tau))d\tau\bigg)\\
&&-\varphi_{q}\bigg(\int_{0}^{s}a(\tau)f(\tau,v(\tau))d\tau\bigg)\Bigg|ds\\
&\leq&\int_{0}^{1}\mathcal{K}(t,s)(q-1)M^{q-2}\Bigg(\int_{0}^{s}a(\tau)\Big|f(\tau,u(\tau))-f(\tau,v(\tau))\Big|d\tau\Bigg)ds\\
&\leq& L\|u-v\|(q-1)M^{q-2}\Bigg(\int_{0}^{1}a(\tau)d\tau \Bigg)\Bigg(\int_{0}^{1}\Phi(s)ds \Bigg)\\
&=& L_{1}\|u-v\|,
\end{eqnarray*}
where
\begin{equation*}
\begin{split}
L_{1}&=L (q-1)M^{q-2}\Bigg(\int_{0}^{1}a(\tau)d\tau \Bigg)\Bigg(\int_{0}^{1}\Phi(s)ds \Bigg) \\
& = \frac{L (q-1)M^{q-2}(\alpha+1)}{\Gamma(\alpha+1)} \int_{0}^{1}a(\tau)d\tau.\\
\end{split}
\end{equation*}
By definition of $L$, we have $0<L_{1}<1.$
Then
\[\|\mathcal{A}u-\mathcal{A}v\|\leq L_{1}\|u-v\|.\]
By virtue of Theorem \ref{T3}, it follows that there exists a
unique fixed point for the operator $\mathcal{A}$, which corresponds to
the unique solution for problem \eqref{eq1} and \eqref{eq2}.
\end{proof}

\begin{theorem}\label{th3.4}
Assume that \rm{(H1)}-\rm{(H2)} hold, and $p>2$. In addition, we assume that there exist constants $\mu> 0$, $0 <\sigma<\frac{2}{2-q}$ such that
\begin{equation} \label{eq16}
a(t)f(t,u)\geq \mu \sigma t^{\sigma-1},\ \text{for any}\ (t,u)\in(0,1]\times [0,\infty),
\end{equation}
and
\begin{equation} \label{eq17}
|f(t,u)-f(t,v)|\leq k |u-v|, \ \text{for any}\ t\in[0,1] \ \text{and}\ u,v\in[0,\infty),
\end{equation}
where
$$0<k<\frac{\big(\sigma(q-2)+\alpha\big)\Gamma(\alpha-1)}{(q-1)\mu^{q-2}\big(\sigma(q-2)+\alpha+1\big)\mathfrak{B}\big(\alpha-1, \sigma(q-2)+1\big)}\Bigg(\int_{0}^{1}a(\tau)d\tau\Bigg)^{-1}.$$
Then the boundary value problem \eqref{eq1} and \eqref{eq2} has a unique solution.
\end{theorem}

\begin{proof}
By \eqref{eq16}, we get
\[\int_{0}^{t}a(s)f(s,u(s))ds\geq \mu t^{\sigma},\ \text{for any}\ (t,u)\in[0,1]\times[0,\infty).\]
By \rm{(i)} in Lemma \ref{lem2.7} and \eqref{eq17}, for any $u,v\in\mathcal{B}$, we have
\begin{eqnarray*}
|\mathcal{A}u(t)-\mathcal{A}v(t)|&=&\Bigg|\int_{0}^{1}\mathcal{K}(t,s)\Bigg(\varphi_{q}\bigg(\int_{0}^{s}a(\tau)f(\tau,u(\tau))d\tau\bigg)\\
&&-\varphi_{q}\bigg(\int_{0}^{s}a(\tau)f(\tau,v(\tau))d\tau\bigg)\Bigg)ds\Bigg|\\
&\leq& \int_{0}^{1}\mathcal{K}(t,s)\Bigg|\varphi_{q}\bigg(\int_{0}^{s}a(\tau)f(\tau,u(\tau))d\tau\bigg)\\
&&-\varphi_{q}\bigg(\int_{0}^{s}a(\tau)f(\tau,v(\tau))d\tau\bigg)\Bigg|ds\\
&\leq&\int_{0}^{1}\mathcal{K}(t,s)(q-1)(\mu s^{\sigma})^{q-2}\Bigg(\int_{0}^{s}a(\tau)\Big|f(\tau,u(\tau))-f(\tau,v(\tau))\Big|d\tau\Bigg)ds\\
&\leq& (q-1)\mu^{q-2}k\|u-v\|\bigg(\int_{0}^{1}a(\tau)d\tau\bigg)\int_{0}^{1}\mathcal{K}(t,s) s^{\sigma(q-2)}ds\\
&\leq& (q-1)\mu^{q-2}k\|u-v\|\bigg(\int_{0}^{1}a(\tau)d\tau\bigg)\int_{0}^{1}\big[\mathcal{G}(s,s)+\mathcal{H}(s,s)\big]s^{\sigma(q-2)}ds\\
&=&(q-1)\mu^{q-2}k\|u-v\|\bigg(\int_{0}^{1}a(\tau)d\tau\bigg)\\
&&\times\Bigg(\frac{\mathfrak{B}\big(\alpha, \sigma(q-2)+1\big)}{\Gamma(\alpha)}+\frac{\mathfrak{B}\big(\alpha-1, \sigma(q-2)+1\big)}{\Gamma(\alpha-1)}\Bigg)\\
&=&\frac{(q-1)\mu^{q-2}k\big(\sigma(q-2)+\alpha+1\big)}{\big(\sigma(q-2)+\alpha\big)\Gamma(\alpha-1)}\bigg(\int_{0}^{1}a(\tau)d\tau\bigg)\\
&&\times\mathfrak{B}\big(\alpha-1, \sigma(q-2)+1\big)\|u-v\|\\
&=&L\|u-v\|,
\end{eqnarray*}
where
\begin{equation*}
L=\frac{(q-1)\mu^{q-2}k\big(\sigma(q-2)+\alpha+1\big)}{\big(\sigma(q-2)+\alpha\big)\Gamma(\alpha-1)}\bigg(\int_{0}^{1}a(\tau)d\tau\bigg)\mathfrak{B}\big(\alpha-1, \sigma(q-2)+1\big).
\end{equation*}
By definition of $k$, we have $0<L<1.$
Then
\[\|\mathcal{A}u-\mathcal{A}v\|\leq L\|u-v\|.\]
This implies that $\mathcal{A} : \mathcal{B}\rightarrow \mathcal{B}$ is a contraction mapping. By Theorem \ref{T3}, we get that $\mathcal{A}$ has a unique fixed point in $\mathcal{B}$,
which is a solution of the problem \eqref{eq1} and \eqref{eq2}.
\end{proof}

 \section{Examples}
\begin{exmp}
Consider the nonlinear boundary value problem

\begin{equation}\label{eq5.1}
       \begin{cases}\Big(\varphi_{\frac{3}{2}}\big(D^{\frac{5}{2}}u(t)\big)\Big)^{\prime} +\frac{1}{2}te^{t}\ln(u+1) = 0,\  t \in (0,1),\\
        D^{\frac{5}{2}}u(0)= u^{\prime} (0) =u^{\prime \prime}(0) =0,\ u(1)+ u^{\prime }(1)=u^{\prime}(\eta),
       \end{cases}
       \end{equation}
where  $ f(t,u) =\frac{1}{2}t\ln(u+1),\ a(t)=e^{t},\ \alpha = \frac{5}{2},\ p=\frac{3}{2},\ \eta\in(0,1),$
and then $q=3,\ f \in C([0,1]\times [0,\infty) ,[0,\infty)).$
By taking $\nu=1$, we obtain
\[L=\max\{f(t,u) / \ t\in [0,1], u\in[0,1] \}=\frac{1}{2}\ln(2),\]
and
\[ L^{q-1}\varphi_{q}\bigg(\int_{0}^{1}a(\tau)d\tau\bigg)\int_{0}^{1}\Phi(s)ds\approx0.372<\nu=1.\]
By means of Theorem \ref{th3.3}, the boundary value problem \eqref{eq5.1}
has at least one positive solution.
\end{exmp}

\begin{exmp}
As a second example we consider the following boundary value problem

\begin{equation}\label{eq5.2}
  \begin{cases}\Big(\varphi_{\frac{3}{2}}\big(D^{\frac{13}{5}}u(t)\big)\Big)^{\prime} +te^{-t}\sin^{2}u = 0,\  t \in (0,1),\\
        D^{\frac{13}{5}}u(0)= u^{\prime} (0) =u^{\prime \prime}(0) =0,\ u(1)+ u^{\prime }(1)=u^{\prime}(\eta),
       \end{cases}
\end{equation}
where  $ f(t,u) =e^{-t}\sin^{2}u,\ a(t)=t,\ \alpha = \frac{13}{5},\ p=\frac{3}{2},\eta\in(0,1),$
and then $q=3,\ f \in C([0,1]\times [0,\infty) ,[0,\infty)).$
\\
Taking the nonnegative function $k(t)=e^{-t}$, then $k\in C[0,1]$ and $f(t,u)\leq k(t).$
Choosing $L=2$, for any $t\in [0,1]$ and $u,v\in[0,\infty)$, we have
\begin{equation*}
\begin{split}
|f(t,u)-f(t,v)|&=e^{-t}|\sin^{2}u-\sin^{2}v|\\
& \leq 2|u-v|\\
&=L |u-v|,
\end{split}
\end{equation*}
and
\begin{eqnarray*}
 \frac{\Gamma(\alpha+1)}{(\alpha+1)(q-1)}\Bigg(\int_{0}^{1}a(t)dt\Bigg)^{-1}\Bigg(\int_{0}^{1}a(t)k(t)dt\Bigg)^{2-q} &=&\frac{13}{18}\Gamma\bigg(\frac{13}{5}\bigg)\big(1-2e^{-1}\big)^{-1}\\
& \approx & 3.90744\\
&>& 2=L.
\end{eqnarray*}
From Theorem \ref{th3.5}, the boundary value problem \eqref{eq5.2} has a
unique solution.
\end{exmp}

\begin{exmp}
Let the following boundary value problem

\begin{equation}\label{eq5.3}
 \begin{cases}\Big(\varphi_{\frac{7}{2}}\big(D^{\frac{5}{2}}u(t)\big)\Big)^{\prime} +\frac{1}{160}t\sqrt{t} (348+\sqrt{u}+t) = 0,\  t \in (0,1),\\
        D^{\frac{5}{2}}u(0)= u^{\prime} (0) =u^{\prime \prime}(0) =0,\ u(1)+ u^{\prime }(1)=u^{\prime}(\frac{1}{2}),
       \end{cases}
\end{equation}
where $ f(t,u) =\frac{1}{400}(348+\sqrt{u}+t) ,\ a(t)=\frac{5}{2}t\sqrt{t},\ \alpha = \frac{5}{2},\ p=\frac{7}{2},\eta=\frac{1}{2},$
and then $q=\frac{7}{5},\ f \in C([0,1]\times [0,\infty) ,[0,\infty)).$
By a simple computation, we obtain

\begin{equation*}
\begin{split}
\Lambda_{1}= &\Bigg(\varphi_{q}\bigg(\int_{0}^{1}a(\tau)d\tau\bigg) \int_{0}^{1}\Phi(s)ds\Bigg)^{-1}\\
= & \Bigg(\varphi_{\frac{7}{5}}\bigg(\int_{0}^{1}\frac{5}{2}\tau\sqrt{\tau}d\tau\bigg) \int_{0}^{1}\frac{(\frac{5}{2}-s)(1-s)^{\frac{1}{2}}}{\Gamma(\frac{5}{2})}ds\Bigg)^{-1}\\
= &\frac{15\sqrt{\pi}}{28}\\
\approx &0.94952,\\
\Lambda_{2}= &\Bigg(\gamma \int_{0}^{\rho}\Phi(s)\varphi_{q}\bigg(\int_{0}^{s}a(\tau)d\tau\bigg)ds\Bigg)^{-1}\\
= & \Bigg(\gamma \int_{0}^{\rho}\Phi(s)\varphi_{\frac{7}{5}}\bigg(\int_{0}^{s}\frac{5}{2}\tau\sqrt{\tau}d\tau\bigg)ds\Bigg)^{-1}\\
=& \Bigg(\gamma \int_{0}^{\rho}s\Phi(s)ds\Bigg)^{-1}
\end{split}
\end{equation*}
\begin{equation*}
\begin{split}
 = & \Bigg(\frac{\big(1-\eta^{\alpha-2}\big)\big(1-\rho^{\alpha-1}\big)}{(1-\alpha)\Gamma(\alpha)}\Bigg((1-\rho)^{\alpha-1}\Bigg(\rho(\alpha-\rho)
+\frac{(\alpha-2\rho)(1-\rho)}{\alpha}\\
&-\frac{2(1-\rho)^{2}}{\alpha(\alpha+1)}\Bigg)+\frac{2}{\alpha(\alpha+1)}-1\Bigg)\Bigg)^{-1}\\
= & \Bigg(\frac{8}{9\sqrt{\pi}}\bigg(1-\frac{1}{\sqrt{2}}\bigg)\big(1-\rho^{\frac{3}{2}}\big)\Bigg(\frac{27}{35}-(1-\rho)^{\frac{3}{2}}\Bigg(\rho\bigg(\frac{5}{2}-\rho\bigg)
+\frac{2}{5}\bigg(\frac{5}{2}-2\rho\bigg)(1-\rho)\\
&-\frac{8}{35}(1-\rho)^{2}\Bigg)\Bigg)\Bigg)^{-1}
\end{split}
\end{equation*}

Choosing $ M_{1}=\Lambda_{1}$, $M_{2}=\Lambda_{2}$, $\rho_{1}=\frac{1}{120}$ and $\rho_{2}=1$. With the use of the Mathematica software, we easy to check that $M_{2}\rho_{1}=\frac{\Lambda_{2}}{120}<M_{1}\rho_{2}=\frac{15\sqrt{\pi}}{28}$ for all $\rho\in\big[\frac{1}{5},\frac{4}{5}\big]$. By a simple computation, we obtain $M_{1}^{\frac{5}{2}}=\frac{225\sqrt{\frac{15}{7}}\pi^{\frac{5}{4}}}{1568}\approx 0.87855$, and  $\big(\frac{M_{2}}{120}\big)^{\frac{5}{2}}\leq 0.86233$ for all $\rho\in\big[\frac{1}{5}$.
Again, we see that $f$ satisfies the following relations:

\begin{equation*}
\begin{split}
f(t,u)=&\frac{1}{400}(348+\sqrt{u}+t)\leq 0.875<\varphi_{\frac{7}{2}}\big(M_{1}\rho_{2}\big)=M_{1}^{\frac{5}{2}}\approx 0.87855,\ t\in[0,1],\ u\in[0,1],\\
f(t,u)=&\frac{1}{400}(348+\sqrt{u}+t)\geq 0.87>\varphi_{\frac{7}{2}}\big(M_{2}\rho_{1}\big)=\Big(\frac{M_{2}}{120}\Big)^{\frac{5}{2}},\ t\in[0,\rho],\ u\in\bigg[\frac{\gamma}{120},\frac{1}{120}\bigg].
\end{split}
\end{equation*}
So, all the assumptions of Theorem \ref{th3.1} are satisfied. With the use of Theorem \ref{th3.1}, the fractional boundary value problem \eqref{eq5.3} has at least one positive solution $u$ such that $\frac{1}{120}< \|u\|<1.$

\end{exmp}

\end{document}